\documentclass[10pt,a4paper]{article}
\usepackage[utf8]{inputenc}
\usepackage[english]{babel}

\usepackage{amsmath}
\usepackage{amsfonts}
\usepackage{amssymb}
\usepackage{amsthm}

\usepackage{graphicx}
\usepackage[margin=2cm]{geometry}
\usepackage{paralist}

\usepackage{csquotes,etoolbox,keyval,ifthen,url,babel}

\usepackage{hyperref}

\hypersetup{
    colorlinks=true,
    linkcolor=red,
    filecolor=magenta,      
    urlcolor=cyan,
    breaklinks=true
            } 

\usepackage{amsthm}

\usepackage{amsthm}

\usepackage{chngcntr}

\usepackage{caption}
\usepackage{subcaption}
\usepackage{wrapfig}

\usepackage{booktabs}

\newtheorem{theorem}{Theorem}
\newtheorem{lemma}{Lemma}

\newtheorem{example}{Example}

\counterwithin{theorem}{section}
\counterwithin{definition}{section}
\counterwithin{example}{section}
\counterwithin{lemma}{section}
\counterwithin{remark}{section}

\DeclareMathOperator{\dif}{d}

\allowdisplaybreaks
\usepackage{mathtools}
\usepackage[multiple]{footmisc} 

\title{On construction of a global numerical solution for a semilinear singularly--perturbed reaction diffusion boundary value problem}
\author{Samir Karasulji\'c\footnote{corresponding author}\:$^{,\thinspace}$\footnote{University of Tuzla, Faculty of Sciences and Mathematics, 
		                Univerzitetska br. 4, 75 000 Tuzla, Bosnia and Herzegovina, email:\texttt{samir.karasuljic@untz.ba}} 
	     and Hidajeta Ljevakovi\'c\footnote{University of Tuzla, Faculty of Sciences and Mathematics, Univerzitetska br. 4, 75 000 Tuzla, Bosnia and Herzegovina, email: \texttt{hidajeta1993@gmail.com}}}
\date{}

\begin{document}
\maketitle
\begin{abstract}
	A class of different schemes for the numerical solving of semilinear singularly—perturbed reaction—diffusion boundary--value problems was constructed. The stability of the difference schemes was proved, and the existence and uniqueness of a numerical solution were shown. After that,  the uniform convergence with respect to a perturbation parameter $\varepsilon$ on a modified Shishkin mesh of order 2 has been proven. For such a discrete solution, a global solution based on a linear spline was constructed, also the error of this solution is in expected boundaries. Numerical experiments at the end of the paper, confirm the theoretical results. The global solutions based on a natural cubic spline, and the experiments with Liseikin, Shishkin and modified Bakhvalov meshes are included in the numerical experiments as well.
\end{abstract}

\section{Introduction}

We consider the semilinear boundary--value singularly--perturbed problem 
\begin{subequations}
	\begin{equation}\label{problem1}
	\varepsilon^2y''-f(x,y)=0,\:x\in(0,1),\;y(0)=y(1)=0,
	\end{equation}
with the condition 
    \begin{equation}\label{problem2}
    \frac{\partial f(x,y)}{\partial y}:=f_y\geqslant m>0,
    \end{equation}	
\end{subequations}
where $0<\varepsilon<<$ is a small perturbation parameter, and  $m$ is a positive constant, $f$ is a nonlinear function, $f(x,y)\in C^{k}\left([0,1]\times\mathbb{R}\right),$ $k\geqslant 2.$ The problem \eqref{problem1} under the condition \eqref{problem2} has a unique solution, (see Lorenz \cite{lorenz1982stability}). It's a well-known fact in theory that the exact solution to \eqref{problem1}--\eqref{problem2} has two exponential boundary layers, i.e. near the end points $x=0$ and $x=1.$   

Differential equations like \eqref{problem1} and similar occur in mathematical modeling of many problems in physics, chemistry, biology, engineering sciences, economics and even  social sciences.  Numerical solutions of singularly--perturbed boundary--value problems obtained by some classical methods are usually useless. That is because the exact solutions of the singularly--perturbed boundary--value problems depend on the perturbation parameter $\varepsilon,$ but classical methods don't take in account the influence of the perturbation parameter.  The singularly--perturbed problems require special developed numerical methods in order to obtain the accuracy, which is uniform respect to the parameter $\varepsilon.$   Numerical methods that act uniformly well for all the values of the singular perturbation parameter  are called $\varepsilon$-uniformly convergent numerical methods. 

Many authors have worked on the numerical solution of the problem \eqref{problem1}--\eqref{problem2} with different assumptions about the function $f,$  as well as more general nonlinear problems. There were many constructed $\varepsilon$--uniformly convergent difference schemes of order 2 and higher (Herceg \cite{herceg1990}, Herceg, Surla and Rapaji{\'c} \cite{herceg1991}, Herceg and Miloradovi{\'c} \cite{herceg2003}, Herceg and Herceg \cite{herceg2003a}, Kopteva and Lin\ss\, \cite{kopteva2001}, Kopteva and Stynes \cite{kopteva2001robust,kopteva2004}, Kopteva, Pickett and Purtill \cite{kopteva2009}, Lin\ss, Roos and Vulanovi{\'c} \cite{linss2000uniform}, Sun and Stynes \cite{stynes1996}, Stynes and Kopteva \cite{stynes2006numerical}, Surla and Uzelac \cite{surla2003}, Vulanovi{\'c} \cite{vulanovic1983numerical, vulanovic1989, vulanovic1991second, vulanovic1993, vulanovic2004}, etc.

In the paper \cite{boglaev1984approximate} Boglaev introduced a new method for the numerical solving of the problem \eqref{problem1}--\eqref{problem2}, using the representation of the exact problem to \eqref{problem1}--\eqref{problem2} via the Green function. In this paper we use this method to construct a new different scheme.

The author's results in the numerical solving of the problem \eqref{problem1}--\eqref{problem2} and others results can be seen in 
\cite{samir2011scheme},  \cite{samir2011skoplje}, \cite{samir2012class}, \cite{samir2013collocation}, 
\cite{samir2015uniformlyconvergent},  \cite{samir2015construction},  \cite{samir2012uniformnly},   \cite{samir2015uniformly},  
\cite{karasuljic2019class},   
\cite{samir2018uniformly}, \cite{samir2017construction}, 
\cite{liseikin2020numerical}, \cite{liseikin2019rules}.

\section{Theoretical background}
The estimates of solution's derivatives are a very important tool in the analysis of numerical methods considering the singularly--perturbed boundary--value problems. The construction of layer--adapted meshes is based on these estimates, also in the sequel they will be used in the analysis of the consistency.
Bearing in mind the above, we state the following theorem about a decomposition of the solution $y$ to a layer component $s$ and a regular component $r$ and the appropriate estimates.

\begin{theorem} {\rm\cite{vulanovic1983numerical} \label{theoremDecomposition}
	} The solution $y$ to problem \eqref{problem1}--\eqref{problem2} can be represented in the following way:
	\begin{equation*}
		y=r+s,
	\end{equation*}
	where for $j=0,1,...,k+2$ and $x\in[0,1]$ we have  that
	\begin{equation}
		\left|r^{(j)}(x)\right|\leq  C,
		\label{regularna}
	\end{equation}
	and
	\begin{equation}
		\left|s^{(j)}(x)\right|\leq  C \varepsilon^{-j}\left(e^{-\frac{x}{\varepsilon}\sqrt{m}}+e^{-\frac{1-x}{\varepsilon}\sqrt{m}}\right).
		\label{slojna}
	\end{equation}
	\label{teorema1}
\end{theorem}

\subsection{Layer--adapted mesh}\label{mreze}
It's a well--known fact that the exact solution to problems like \eqref{problem1}--\eqref{problem2} changes rapidly near the end points $x=0$ and $x=1.$ Many meshes have been constructed for the numerical solving problems that have a layer or layers of an exponential type. In the present paper we shall use three different meshes. We will get these meshes  $0=x_0<x_1<\ldots<x_N=1,$ by using appropriate generating functions, i.e. $x_i=\psi(i/N).$ The generating function are constructed as follows. 

Let $N+1$ be the number of mesh points, $q\in(0,1/2)$ mesh parameter. Define the Shishkin mesh transition point by
\begin{equation}\label{shihskinTransitionPoint}
	\lambda:=\min\left\{\frac{2\varepsilon\ln N}{\sqrt{m}},\frac{1}{4}\right\} .
\end{equation}
The first mesh we will use in the sequel is a modified Shishkin mesh proposed by Vulanovi\'c \cite{vulanovic2001higher}. The generating function for this mesh is 
\begin{equation}\label{meshShishkin2}
	\psi(t)=\begin{cases}
		4\lambda t,\quad  t\in[0,1/4],\\
		p(t-1/4)^3+4\lambda t, \quad t\in[1/4,1/2],\\
		1-\psi(1-t),\quad t\in[1/2,1],
	\end{cases}  
\end{equation}
where $p$ is chosen so that $\psi(1/2)=1/2,$ i.e. $p=32 (1-4\lambda).$ Note that $\psi\in C^{1}[0,1]$ with $\Vert\psi'\Vert_{\infty}\leqslant C,$ $\Vert\psi''\Vert_{\infty}\leqslant C.$ Therefore the mesh size $h_i=x_{i+1}-x_i,\,i=0,\ldots N-1$ satisfy (see \cite{linss2012approximation})
\begin{eqnarray}
	h_i=\int_{i/N}^{(i+1)N}{\psi'(t)\dif t}\leqslant CN^{-1},\quad |h_{i+1}-h_i|=\left|\int_{(i-1)/N}^{i/N}{\int^{t+1/N}_{t}{\psi''(s) \dif s}}\right|\leqslant CN^{-2}.
\end{eqnarray}
The second mesh is the Shishkin mesh \cite{shishkin1988grid}. The generating function for this mesh is 
\begin{equation}\label{meshShishkin1}
\psi(t)=\begin{cases}
4\lambda  t,\quad t\in[0,1/4] \\
\lambda+2(1-2\lambda)(t-1/4),\quad t\in[1/2,1/4],\\
1-\psi(1-t),\quad t\in[1/2,1].
\end{cases}
\end{equation}
The third mesh is the modified Bakhvalov mesh also proposed by Vulanovi\'c \cite{vulanovic1983numerical}. The generating function for this mesh is 

\begin{equation}\label{meshVulanovic1}
\psi(t)=\begin{cases}
\mu(t):=\frac{a\varepsilon t}{q-t},\quad t\in[0,\alpha],\\
\mu(\alpha)+\mu'(\alpha)(t-\alpha),\quad t\in[\alpha ,1/2],\\
1-\psi(1-t),\quad t\in[1/2,1],
\end{cases}
\end{equation}
where $a$ and $q$ are constants, independent of $\varepsilon,$ such that $q\in(0,1/2),\:a\in(0,q/\varepsilon),$ and additionally $a\sqrt{m}\geqslant 2.$ The parameter 
$\alpha$ is the abscissa of the contact point of the tangent line from $(1/2,1/2)$ to $\mu (t),$ and its value is 
\[\alpha =\frac{q-\sqrt{aq\varepsilon(1-2q+2a\varepsilon)}}{1+2a\varepsilon }.\]

The fourth mesh proposed by Liseikin \cite{liseikin2018grid, liseikin2019compact}, and we will use 	its modification from \cite{liseikin2020numerical}. The generating function for this mesh is 

\begin{equation}\label{meshLiseikin}
\psi( t,\varepsilon,a,k)=
\left\{
\begin{array}{ll}
\displaystyle c_1\varepsilon^k ((1-d t) ^{-1/a}-1)\;, & 0\leqslant t \leqslant 1/4 \;,\\[4mm]
\displaystyle c_1\Bigl[\varepsilon^{kan/(1+na)}-\varepsilon^k+d\frac{1}{a}\varepsilon^{ka(n-1)/(1+na)}(t-1/4 )+\\
\frac{1}{2}d^2\frac{1}{a}\Bigl(\frac{1}{a}+1\Bigr)\varepsilon^{ka(n-2)/(1+na)}(t -1/4 )^2+c_0(t -1/4 )^{3}\Bigr]\;, & 1/4  \leqslant  t \leqslant 1/2\;,\\
1-\psi(1-t ,\varepsilon,a,k)\;,& 1/2\leqslant t \leqslant 1\;,
\end{array}
\right .
\end{equation}
where $d=(1-\varepsilon^{ka/(1+na)})/(1/4),$ $a$ is a positive constant subject to $a\geq m_1 > 0$, and $a=1,$  $c_0>0$, $n=2,$ $k=1,$ $c_0=0,$   and
$
\tfrac{1}{c_1}=2\left[\varepsilon^{kan/(1+na)}-\varepsilon^k+\tfrac{d}{4a} \varepsilon^{ka(n-1)/(1+na)}\right.
\left.+ \tfrac{d^2}{2}\tfrac{1}{a}\Bigl(\frac{1}{a}+1\Bigr)\varepsilon^{ka(n-2)/(1+na)}(1/4 )^2+c_0(1/4)^{3}\right]
$
is chosen here.

\section{Difference scheme}
We will consider an arbitrary mesh with mesh points \[0=x_0<x_1<\ldots<x_N=1,\] 
and let it be $h_i=x_{i+1}-x_i,\:i=0,1,\ldots,N-1.$ 
In constructing a new difference scheme for the problem \eqref{problem1}--\eqref{problem2} we use the following scheme from Boglaev \cite{boglaev1984approximate}
\begin{multline}\label{schema0}
\frac{\beta}{\sinh(\beta h_{i-1})}y_{i-1}-\left(\frac{\beta}{\tanh(\beta h_{i-1})}+\frac{\beta}{\tanh(\beta h_i)}\right)y_i +\frac{\beta}{\sinh(\beta h_i)}y_{i+1}\\
    =\frac{1}{\varepsilon^2}\left[\int_{x_{i-1}}^{x_{i}}{u^{II}_{i-1}\psi(s,y)\dif s} +\int_{x_{i}}^{x_{i+1}}{u^{I}_{i}\psi(s,y)\dif s}\right],\:
    i=1,2,\ldots,N-1,\:y_0=y_N=0,
\end{multline}
where 
\begin{equation}
\psi(s,y)=f(s,y)-\gamma y,\:\beta=\frac{\sqrt{\gamma}}{\varepsilon}.
\end{equation}
We can't calculate the integrals in \eqref{schema0}  because we don't know the exact solution $y$ to the problem \eqref{problem1}--\eqref{problem2}. The next step is to approximate  the function $\psi$ by a constant value. Approximations of the function $\psi$ are
\begin{align}
 \psi^{-}_i&=(1-t)\psi(x_{i-1},y(x_{i-1}))+t\psi(x_i,y(x_i)),\:x\in[x_{i-1},x_i],\label{approximation1}\\
 \psi^{+}_i&=t\psi(x_{i},y(x_{i}))+(1-t)\psi(x_{i+1},y(x_{i+1})),\:x\in[x_{i},x_{i+1}],\:t\in[0,1].\label{approximation2}
\end{align}
By using the approximations \eqref{approximation1}, \eqref{approximation2} into \eqref{schema0}, after calculating the integrals and some computing, and taking in account that
\begin{align*}
\int_{x_{i-1}}^{x_{i}}{u^{II}_{i-1}\dif s}=& \frac{\cosh(\beta h_{i-1})-1}{\beta\sinh(\beta h_{i-1})},     \\
\int_{x_{i}}^{x_{i+1}}{u^{I}_{i}\dif s}=&  \frac{\cosh(\beta h_{i})-1}{\beta\sinh(\beta h_{i})},
\end{align*}
we get the difference scheme
\begin{multline}
   \frac{(1-t)\cosh(\beta h_{i-1})+t}{\sinh(\beta h_{i-1})}(\overline{y}_{i-1}-\overline{y}_i)
    - \frac{(1-t)\cosh(\beta h_{i})+t}{\sinh(\beta h_{i})}(\overline{y}_{i}-\overline{y}_{i+1})\\
    -\frac{(1-t)f_{i-1}+tf_i}{\gamma}\cdot\frac{\cosh(\beta h_{i-1})-1}{\sinh(\beta h_{i-1})}
      -\frac{tf_i+(1-t)f_{i+1}}{\gamma}\cdot\frac{\cosh(\beta h_{i})-1}{\sinh(\beta h_{i})}=0,
   \label{shema1} 
\end{multline}
where $f_{k}=f(x_{k},\overline{y}_{k}),\,k\in\{i-1,i,i+1\},$ and $t\in[0,1].$

The previous form of the difference scheme can be written in the following form
\begin{multline*}
   \frac{(1-t)\cosh(\beta h_{i-1})+1-t+2t-1}{\sinh(\beta h_{i-1})}(\overline{y}_{i-1}-\overline{y}_i)
    - \frac{(1-t)\cosh(\beta h_{i})+1-t+2t-1}{\sinh(\beta h_{i})}(\overline{y}_{i}-\overline{y}_{i+1})\\
    -\frac{(1-t)f_{i-1}+(1-t+2t-1)f_i}{\gamma}\cdot\frac{\cosh(\beta h_{i-1})-1}{\sinh(\beta h_{i-1})}
      -\frac{(1-t+2t-1)f_i+(1-t)f_{i+1}}{\gamma}\cdot\frac{\cosh(\beta h_{i})-1}{\sinh(\beta h_{i})}=0,
\end{multline*}
and finally

\begin{align}
  &(1-t)\left[\frac{\cosh(\beta h_{i-1})+1}{\sinh(\beta h_{i-1})}(\overline{y}_{i-1}-\overline{y}_{i})
             -\frac{\cosh(\beta h_{i})+1}{\sinh(\beta h_{i})}(\overline{y}_{i}-\overline{y}_{i+1})\right.\nonumber\\
  &\hspace{4.5cm}            -\left.\frac{f_{i-1}+f_{i}}{\gamma}\cdot\frac{\cosh(\beta h_{i-1})-1}{\sinh(\beta h_{i-1})}
              -\frac{f_{i}+f_{i+1}}{\gamma}\cdot\frac{\cosh(\beta h_{i-1})-1}{\sinh(\beta h_{i-1})}\right]\nonumber\\
  &+(2t-1)\left[ \frac{1}{\sinh(\beta h_{i-1})}(\overline{y}_{i-1}-\overline{y}_{i})
                -\frac{1}{\sinh(\beta h_{i})}(\overline{y}_{i}-\overline{y}_{i+1})\right.\nonumber \\           
  &\hspace{4.5cm} -\left.\frac{f_{i}}{\gamma}\cdot\frac{\cosh(\beta h_{i-1})-1}{\sinh(\beta h_{i-1})}
              -\frac{f_{i}}{\gamma}\cdot\frac{\cosh(\beta h_{i-1})-1}{\sinh(\beta h_{i-1})}\right]=0,\:i=1,\ldots, N-1. \label{shema2}                         
\end{align}

\section{Stability}
The difference scheme \eqref{shema2} generates a nonlinear system. A goal of this section is to show that this  system has a unique solution. We are going to construct a discrete operator $T,$ and show that the discrete operator $T$ is inverse-monotone as well, which implies that our numerical method is stable, and the numerical solution exist and it is a unique. 

Let us set the discrete operator 
\begin{equation}\label{operator1}
Tu=(Tu_0,\,Tu_1,\,\ldots,Tu_N)^T,
\end{equation}
where 
\begin{align}
Tu_0=&-u_0\nonumber\\
Tu_i=& \frac{\gamma}{\frac{\cosh(\beta h_{i-1})-1}{\sinh(\beta h_{i-1})}+\frac{\cosh(\beta h_{i})-1}{\sinh(\beta h_{i})}} \nonumber\\
      &\left\{ (1-t)\left[\frac{\cosh(\beta h_{i-1})+1}{\sinh(\beta h_{i-1})}(\overline{y}_{i-1}-\overline{y}_{i})
     -\frac{\cosh(\beta h_{i})+1}{\sinh(\beta h_{i})}(\overline{y}_{i}-\overline{y}_{i+1})\right.\right.\nonumber\\
     &\hspace{3.5cm}            -\left.\frac{f_{i-1}+f_{i}}{\gamma}\cdot\frac{\cosh(\beta h_{i-1})-1}{\sinh(\beta h_{i-1})}
     -\frac{f_{i}+f_{i+1}}{\gamma}\cdot\frac{\cosh(\beta h_{i-1})-1}{\sinh(\beta h_{i-1})}\right]\nonumber\\
     &+(2t-1)\left[ \frac{1}{\sinh(\beta h_{i-1})}(\overline{y}_{i-1}-\overline{y}_{i})
     -\frac{1}{\sinh(\beta h_{i})}(\overline{y}_{i}-\overline{y}_{i+1})\right.\nonumber \\           
     &\hspace{3.5cm} -\left.\left.\frac{f_{i}}{\gamma}\cdot\frac{\cosh(\beta h_{i-1})-1}{\sinh(\beta h_{i-1})}
     -\frac{f_{i}}{\gamma}\cdot\frac{\cosh(\beta h_{i-1})-1}{\sinh(\beta h_{i-1})}\right]\right\}=0,\:i=1,\ldots, N-1,\label{operator2} \\
Tu_N=&-u_N \nonumber    
\end{align}
Obviously, it is hold 
\begin{equation}\label{operator3}
T\overline{y}=0,
\end{equation}
where $\overline{y}=(\overline{y}_0,\,\overline{y}_1,\ldots,\overline{y}_N)^T$ the numerical solution of the problem \eqref{problem1}--\eqref{problem2}, obtained by using the difference scheme \eqref{shema2}. Now, we can state and prove the theorem of stability. 

\begin{theorem}
The discrete problem \eqref{operator1}--\eqref{operator3} has a unique solution $\overline{y}$ for $\gamma \geqslant f_y.$  Moreover, for every $v,\;w\in \mathbb{R}^{N+1}$ we have the following stability inequality
\begin{equation}
\Vert v-w\Vert\leqslant C\Vert Tv-Tw\Vert.\label{thStab1}
\end{equation}
\end{theorem}

\begin{proof}
	We use a well known technique from \cite{vulanovic1993} to prove the first statement of the theorem. The proof of existence and uniqueness of the solution of the discrete problem $T y=0$ is based on the proof of the relation: $\Vert T y \Vert\leqslant C,$ where $T'$ is the Fr\' echet derivative of $T.$ The Fr\' echet derivative $H:=T'( y )$ is a tridiagonal matrix. Let $H=[h_{ij}].$ The non-zero elements of this tridiagonal matrix are 
\begin{align}
   h_{1,1}=&h_{N+1,N+1}=-1<0,\nonumber \\
   h_{i,i-1}=&\frac{\gamma}{\frac{\cosh(\beta h_{i-1})-1}{\sinh(\beta h_{i-1})}+\frac{\cosh(\beta h_{i})-1}{\sinh(\beta h_{i})}} 
              \left[\frac{(1-t)\cosh(\beta h_{i-1})+t}{\sinh(\beta h_{i-1})}-\frac{\frac{\partial f}{\partial y_{i-1}}}{\gamma}
                \cdot\frac{(1-t)(\cosh(\beta h_{i-1})-1)}{\sinh(\beta h_{i-1})}\right],\nonumber\\
   h_{i,i+1}=& \frac{\gamma}{\frac{\cosh(\beta h_{i-1})-1}{\sinh(\beta h_{i-1})}+\frac{\cosh(\beta h_{i})-1}{\sinh(\beta h_{i})}} 
              \left[\frac{(1-t)\cosh(\beta h_{i})+t}{\sinh(\beta h_{i})}-\frac{\frac{\partial f}{\partial y_{i+1}}}{\gamma}
                \cdot\frac{(1-t)(\cosh(\beta h_{i})-1)}{\sinh(\beta h_{i})}\right],\nonumber\\ 
   h_{i,i}=& \frac{-\gamma}{\frac{\cosh(\beta h_{i-1})-1}{\sinh(\beta h_{i-1})}+\frac{\cosh(\beta h_{i})-1}{\sinh(\beta h_{i})}} 
            \left[\frac{(1-t)\cosh(\beta h_{i-1})+t}{\sinh(\beta h_{i-1})}+\frac{(1-t)\cosh(\beta h_{i})+t}{\sinh(\beta h_{i})} \right.\nonumber\\
         &\hspace{4cm}\left. +t\frac{\frac{\partial f}{\partial y_i}}{\gamma}\cdot \frac{\cosh(\beta h_{i-1})-1}{\sinh(\beta h_{i-1})} 
                             +t\frac{\frac{\partial f}{\partial y_i}}{\gamma}\cdot \frac{\cosh(\beta h_{i})-1}{\sinh(\beta h_{i})}     \right], \:i=2,\ldots,N. \label{stabilnost1}                          
\end{align}
From \eqref{stabilnost1}, it's obvious that
\[h_{i,i-1}>0,\:h_{i,i+1}>0,\:h_{i,i}<0,\]
and
\[\left|h_{i,i}\right|-\left|h_{i,i-1}\right|-\left| h_{i,i+1}\right|\geqslant m,\]
so we can conclude that $H$ is an $M$--matrix, and finally we obtain 
\begin{equation}
\Vert H^{-1}\Vert\leqslant\frac{1}{m}.
\label{diskretni4}
\end{equation}  
Using Hadamard's theorem (\cite[p 137]{ortega2000})  we get that $T$ homeomorphism. Since clearly $\mathbb{R}^{N+1}$ is non--empty and $0$ is the only image of the mapping $T,$ we have that \eqref{operator3}  has a unique solution.

The proof of second statement of the Theorem \ref{thStab1} is based on a part of the proof of Theorem 3 from \cite{herceg1990}. We have that $Tv-Tw=\left(T'\xi\right)^{-1}(v-w)$ for some $\xi=(\xi_0,\xi_1,\ldots,\xi_N)^T\in\mathbb{R}^{N+1}.$ Therefore $v-w=\left( T'\xi\right)^{-1}(Tv-Tw)$ and finally due to inequality \eqref{diskretni4} we have that 
\[\Vert v-w\Vert=\Vert\left( T'\xi\right)^{-1}(Tv-Tw)\Vert\leqslant\frac{1}{m}\Vert Tv-Tw\Vert.\] 
\end{proof}

\section{Uniform convergence}

The difference scheme \eqref{shema1} we can write in the following form 

\begin{multline}
   (1-t)\left[\frac{\cosh(\beta h_{i-1})-1}{\sinh(\beta h_{i-1})}(\overline{y}_{i-1}-\overline{y}_{i})
               -\frac{\cosh(\beta h_{i})-1}{\sinh(\beta h_{i})}(\overline{y}_{i}-\overline{y}_{i+1})  \right]
               +\frac{\overline{y}_{i-1}-\overline{y}_i}{\sinh(\beta h_{i-1})}
                -\frac{\overline{y}_{i}-\overline{y}_{i+1}}{\sinh(\beta h_i)} \\
   -\left.\frac{(1-t)f_{i-1}+tf_i}{\gamma}\cdot\frac{\cosh(\beta h_{i-1})-1}{\sinh(\beta h_{i-1})}
      -\frac{tf_i+(1-t)f_{i+1}}{\gamma}\cdot\frac{\cosh(\beta h_{i})-1}{\sinh(\beta h_{i})}  \right]   =0,\:i=1,\ldots,N-1.
 \label{schema3} 
\end{multline}

In order to prove the Theorem of convergence, we need three estimates given in the next lemmas. 

\begin{lemma}{\rm\cite{samir2015uniformly}}\label{lema1}
  Assume that $\varepsilon \leqslant \frac{C}{N}.$ In the part of the modified Shishkin mesh from Section \ref{mreze} when $x_i,\,x_{i\pm 1}\in[x_{N/4-1},\lambda]\cup[\lambda,1/2],$ we have the following estimate
\begin{equation*}
\left| \frac{\frac{\cosh(\beta h_{i-1})-1}{\sinh(\beta h_{i-1})}(y(x_{i-1})-y(x_i))
              -\frac{\cosh(\beta h_{i})-1}{\sinh(\beta h_{i})}(y(x_{i})-y(x_{i+1})) }
              {\frac{\cosh(\beta h_{i-1})-1}{\sinh(\beta h_{i-1})}+\frac{\cosh(\beta h_{i})-1}{\sinh(\beta h_{i})}}\right|\leqslant\frac{C}{N^2},\:i=N/4,\ldots,N/2-1.
\end{equation*}
\end{lemma}

\begin{lemma}{\rm\cite{samir2015uniformly}}\label{lema2}
  Assume that $\varepsilon \leqslant \frac{C}{N}.$ In the part of the modified Shishkin mesh from Section \ref{mreze} when $x_i,\,x_{i\pm 1}\in[x_{N/4-1},\lambda]\cup[\lambda,1/2],$ we have the following estimate
\begin{equation*}
\left| \frac{\frac{y(x_{i-1})-y(x_i)}{\sinh(\beta h_{i-1})}
              -\frac{y(x_{i})-y(x_{i+1})}{\sinh(\beta h_{i})} }
              {\frac{\cosh(\beta h_{i-1})-1}{\sinh(\beta h_{i-1})}+\frac{\cosh(\beta h_{i})-1}{\sinh(\beta h_{i})}}\right|\leqslant\frac{C}{N^2},\:i=N/4,\ldots,N/2-1.
\end{equation*}
\end{lemma}

\begin{lemma}\label{lema3}
      Assume that $\varepsilon \leqslant \frac{C}{N}.$ In the part of the modified Shishkin mesh from Section \ref{mreze} when $x_i,\,x_{i\pm 1}\in[x_{N/4-1},\lambda]\cup[\lambda,1/2],$ we have the following estimate
\begin{multline}
  \frac{\gamma}{\frac{\cosh(\beta h_{i-1})-1}{\sinh(\beta h_{i-1})}+\frac{\cosh(\beta h_{i})-1}{\sinh(\beta h_{i})}}
   \left| \frac{(1-t)f(x_{i-1},y(x_{i-1}))+tf(x_i,y(x_i))}{\gamma}\cdot\frac{\cosh(\beta h_{i-1})-1}{\sinh(\beta h_{i-1})}\right.\\
      -\left.\frac{tf(x_i,y(x_i))+(1-t)f(x_{i+1},y(x_{i+1}))}{\gamma}\cdot\frac{\cosh(\beta h_{i})-1}{\sinh(\beta h_{i})}\right|
      \leqslant \frac{C}{N^2},\:i=N/4,\ldots,N/2-1.
\end{multline}      
\end{lemma}
\begin{proof}
Taking into consideration the assumption $\varepsilon\leqslant \frac{C}{N},$ the equality $\varepsilon^2y''(x_i)=f(x_i,y(x_i)),$ and the Theorem of decomposition, it is hold   
\begin{align}
&      \frac{\gamma}{\frac{\cosh(\beta h_{i-1})-1}{\sinh(\beta h_{i-1})}+\frac{\cosh(\beta h_{i})-1}{\sinh(\beta h_{i})}}
   \left| \frac{(1-t)f(x_{i-1},y(x_{i-1}))+tf(x_i,y(x_i))}{\gamma}\cdot\frac{\cosh(\beta h_{i-1})-1}{\sinh(\beta h_{i-1})} \right.\nonumber\\
& \hspace{4cm}  -\left.\frac{tf(x_i,y(x_i))+(1-t)f(x_{i+1},y(x_{i+1}))}{\gamma}\cdot\frac{\cosh(\beta h_{i})-1}{\sinh(\beta h_{i})}\right| \nonumber \\
&\quad\leqslant \left| (1-t)f(x_{i-1},y(x_{i-1}))+tf(x_i,y(x_i)) + tf(x_{i},y(x_{i}))+(1-t)f(x_{i+1},y(x_{i+1})) \right|\nonumber\\
&\quad \leqslant \varepsilon^2\left[ (1-t)\left(\left|r''(x_{i-1})\right|+\left|s''(x_{i-1}) \right|\right)
                       +2t\left(\left|s''(x_i)\right|+\left|r''(x_i)\right| \right) 
                       +(1-t)\left(\left|s''(x_{i+1})\right|+\left|r''(x_{i+1})\right| \right)\right]\nonumber\\
&\quad\leqslant C_1\varepsilon^2\left[ (1-t)\left(1+\frac{e^{-\frac{x_{i-1}}{\varepsilon}\sqrt{m}}}{\varepsilon^2}\right) 
                                     +2t\left(1+\frac{e^{-\frac{x_i}{\varepsilon}\sqrt{m}}}{\varepsilon^2}\right)
                                     +(1-t)\left(1+\frac{e^{-\frac{x_{i+1}}{\varepsilon}\sqrt{m}}}{\varepsilon^2} \right)    \right] \nonumber\\                      
&\quad \leqslant C\left( \varepsilon^2+\frac{1}{N^2}\right),\:i=N/4,\ldots,N/2-1.                       
\end{align}
\end{proof}

\begin{theorem}\label{thConvergence1}
  The discrete problem \eqref{operator1}--\eqref{operator3} on the modified Shishkin mesh \eqref{meshShishkin2} from Section \ref{mreze} is uniformly convergent with respect $\varepsilon$ and 
  \begin{equation*}
  	\max_i\left|y(x_i)-\overline{y}_i\right|\leqslant C 
  	\left\{\begin{array}{ll}
  		\left( \ln^2 N\right)/N^2,&i=0,\ldots,N/4-1,\\\\
  		 1/ N^2 ,& i=N/4,\ldots,3N/4,\\\\
  		\left( \ln^2 N \right) / N^2 ,& i=3N/4+1,\ldots,N,
  	\end{array}   
  	\right.
  \end{equation*}
  where $y(x_i)$ is the value of the exact solution, $\overline{y}_i$ is the value of the numerical solution  of the problem \eqref{problem1}--\eqref{problem2} in the mesh point $x_i,$ respectively, and $C>0$ is a constant independent of $N$ and $\varepsilon.$
\end{theorem}

\begin{proof}\ \\
\noindent\textbf{Case $0\leqslant i<N/4-1.$} Here it's hold $h_{i-1}=h_i$ and $h_i=\mathcal{O}(\varepsilon\ln N/N).$ We  have

\begin{align}
&(Ty)_i=\nonumber\\
&  \frac{\gamma}{\frac{\cosh(\beta h_{i-1})-1}{\sinh(\beta h_{i-1})}+\frac{\cosh(\beta h_i)-1}{\sinh(\beta h_i)}}
      \left[ \frac{(1-t)\cosh(\beta h_{i-1})+t}{\sinh(\beta h_{i-1})}\left( y(x_{i-1})-y(x_i)\right)
         - \frac{(1-t)\cosh(\beta h_{i})+t}{\sinh(\beta h_{i})}\left( y(x_{i})-y(x_{i+1})\right) \right.\nonumber\\
& \hspace{3cm} -\frac{(1-t)f(x_{i-1},y(x_{i-1}))+tf(x_i,y(x_i)) }{\gamma}\cdot\frac{\cosh(\beta h_{i-1})-1}{\sinh(\beta h_{i-1})} \nonumber\\
&\hspace{4cm}        -\left.\frac{tf(x_{i},y(x_{i}))+(1-t)f(x_{i+1},y(x_{i+1})) }{\gamma}\cdot\frac{\cosh(\beta h_{i})-1}{\sinh(\beta h_{i})}\right]\nonumber\\
&=\frac{\gamma}{2(\cosh(\beta h_i)-1)}
\left\{  t\left[y(x_{i-1})-2y(x_i)+y(x_{i+1})- \frac{2f(x_i,y(x_i))}{\gamma}(\cosh(\beta h_i)-1)\right]      \right. \nonumber \\
&\hspace{.5cm} +(1-t)\left.\left[\cosh(\beta h_i)\left(y(x_{i-1})-2y(x_i)+y(x_{i+1}) \right)
              -\frac{f(x_{i-1},y(x_{i-1}))+f(x_{i+1},y(x_{i+1}))}{\gamma}\cdot  (  \cosh(\beta h_i)-1 )      \right] \right\}\nonumber\\
&=\frac{\gamma}{2(\cosh(\beta h_i)-1)}
\left\{  t\left[y(x_{i-1})-2y(x_i)+y(x_{i+1})- \frac{2\varepsilon^2 y''(x_i)}{\gamma}(\cosh(\beta h_i)-1)\right]      \right. \nonumber \\
&\hspace{.5cm} +(1-t)\left.\left[\cosh(\beta h_i)\left(y(x_{i-1})-2y(x_i)+y(x_{i+1}) \right)
              -\varepsilon^2\frac{ y''(x_{i-1}) +y''(x_{i+1}) }{\gamma}\cdot (  \cosh(\beta h_i)-1 )      \right] \right\}.\nonumber              
\end{align}
Using Taylor's expansions 
\begin{align*}
&  y(x_{i-1})-2y(x_i)+y(x_{i+1})=y''(x_i)h^2_i+\frac{y^{(iv)}(\xi^{-}_i)+y^{(iv)}(\xi^{+}_i)}{24}h^4_i,\,\xi^{-}_i\in(x_{i-1},x_i),\,\xi^{+}_i\in(x_i,x_{i+1}),\\
&  y''(x_{i-1})+y''(x_{i+1})=2y''(x_i)+\frac{y^{(iv)}(\eta^{-}_i)+y^{(iv)}(\eta^{+}_i)}{2}h^2_i,\,\eta^{-}_i\in(x_{i-1},x_i),\,\eta^{+}_i\in(x_i,x_{i+1}),\\
&  \cosh(\beta h_i)=1+\frac{\beta^2h^2_i}{2}+\mathcal{O}\left(\beta^4h^4_i\right)  
\end{align*}
we get

\begin{align}
  (Ty)_i=& \frac{\gamma\cdot t}{\beta^2 h^2_i+2\mathcal{O}(\beta^4 h^4_i)}
     \left[y''(x_i)h^2_i+\frac{y^{(iv)}(\xi^{-}_i)+y^{(iv)}(\xi^{+}_i)}{24}h^4_i
          -\frac{2\varepsilon^2y''(x_i)}{\gamma}\left(\frac{\beta^2h^2_i}{2}+\mathcal{O}(\beta^4h^4_i) \right)  \right]\nonumber\\
      &+ \frac{\gamma\cdot (1-t)}{\beta^2 h^2_i+2\mathcal{O}(\beta^4 h^4_i)} 
      \left[\left( 1+\frac{\beta^2h^2_i}{2}+\mathcal{O}(\beta^4h^4_i)\right) 
         \left(y''(x_i)h^2_i+\frac{y^{(iv)}(\xi^{-}_i)+y^{(iv)}(\xi^{+}_i)}{24}h^4_i \right)\right.   \nonumber\\
      &\hspace{3.5cm}-\left.\varepsilon^2\frac{2y''(x_i)+\frac{y^{(iv)}(\eta^{-}_i)+y^{(iv)}(\eta^{+}_i)}{2}}{\gamma}
      \left(\frac{\beta^2h^2_i}{2}+\mathcal{O}(\beta^4h^4_i) \right)  \right] \nonumber\\
     =& \frac{\gamma\cdot t}{\beta^2 h^2_i+2\mathcal{O}(\beta^4 h^4_i)}
       \left[ \frac{y^{(iv)}(\xi^{-}_i)+y^{(iv)}(\xi^{+}_i)}{24}h^4_i -\frac{2\varepsilon^2y''(x_i)}{\gamma}\mathcal{O}(\beta^4h^4_i)\right]\nonumber\\
     &+ \frac{\gamma\cdot (1-t)}{\beta^2 h^2_i+2\mathcal{O}(\beta^4 h^4_i)}  
     \left[ \frac{y^{(iv)}(\xi^{-}_i)+y^{(iv)}(\xi^{+}_i)}{24}h^4_i +\left(\frac{\beta^2h^2_i}{2}+\mathcal{O}(\beta^4h^4_i) \right) 
                   \left(y''(x_i)h^2_i+\frac{y^{(iv)}(\xi^{-}_i)+y^{(iv)}(\xi^{+}_i)}{24}h^4_i  \right)   \right.   \nonumber\\
    &-\left.\frac{2\varepsilon^2y''(x_i)}{\gamma}\cdot\mathcal{O}(\beta^4h^4_i)
      +  \varepsilon^2\frac{y^{(iv)}(\eta^{-}_i)+y^{(iv)}(\eta^{+}_i)}{2\gamma}h^2_i\left(\frac{\beta^2h^2_i}{2}+\mathcal{O}(\beta^4h^4_i) \right)  \right],           
\end{align}
and finally
\begin{equation}
     \left|(Ty)_i\right|\leqslant \left( C\ln^2N \right) / N^2 ,\,i=0,1,\ldots,N/4-1.
  \label{thuniform0}
\end{equation}

\noindent\textbf{Case $N/4\leqslant i<N/2.$}   Due to \eqref{schema3} we have the next inequality 

\begin{align}
   \left|(Ty)_i\right|\leqslant& \frac{\gamma}{\frac{\cosh(\beta h_{i-1})-1}{\sinh(\beta h_{i-1})}+\frac{\cosh(\beta h_i)-1}{\sinh(\beta h_i)}}\nonumber\\
  &  \cdot\left[(1-t) \left|\frac{\cosh(\beta h_{i-1})-1}{\sinh(\beta h_{i-1})}(y(x_{i-1})-y(x_i))
              -\frac{\cosh(\beta h_{i})-1}{\sinh(\beta h_{i})}(y(x_{i})-y(x_{i+1}))  \right|\right. \nonumber\\
     &\quad+ \left|\frac{y(x_{i-1})-y(x_i)}{\sinh(\beta h_{i-1})}
              -\frac{y(x_{i})-y(x_{i+1})}{\sinh(\beta h_{i})} \right|   \nonumber\\ 
     &\quad +\left| \frac{(1-t)f(x_{i-1},y(x_{i-1}))+tf(x_i,y(x_i))}{\gamma}\cdot \frac{\cosh(\beta h_{i-1})-1}{\sinh(\beta h_{i-1})}\right|  \nonumber\\
     &\quad +\left.\left| \frac{tf(x_{i},y(x_{i}))+(1-t)f(x_{i+1},y(x_{i+1}))}{\gamma}\cdot \frac{\cosh(\beta h_{i})-1}{\sinh(\beta h_{i})}\right|\right], \nonumber        
\end{align}
and according Lemma \ref{lema1}, Lemma \ref{lema2} and Lemma \ref{lema3} we obtain
\begin{equation}
    \left| (Ty)_i\right|\leqslant  C/ N^2,\,i=N/4,\ldots,N/2-1.
 \label{thuniform1}
\end{equation}

\noindent\textbf{Case $i=N/2.$} This case is trivial, because $h_{N/4-1}=h_{N/4}$ and the influence of the layer component $s$ is negligible. \\

\noindent Collecting \eqref{thuniform0}, \eqref{thuniform1} and taking in account Case $i=N/2,$ we have proven the theorem.  
\end{proof}

\section{Global solution}
In the paper \cite{samir2017construction} a global numerical solution was constructed using a spline in tension, and the authors proved the uniform convergence of order 1 for this solution on the modified Shishkin mesh generated by \eqref{meshShishkin2}. After that they repaired the global numerical solution on $[\lambda,1-\lambda]$ and achieved  the uniform convergence of order 2. That repaired global solution is composed of exponential and linear functions. In the sequel we avoid exponential functions and give a global numerical solution composed by linear functions. We will also include in the numerical experiments a global solution obtained by using a natural cubic spline, because this spline is the lowest degree spline with a continuous second derivative.

\paragraph{Linear spline} Let the global numerical solution to the problem \eqref{problem1}--\eqref{problem2} has the form
\begin{equation}\label{globalSolution1}
   \overline{P}(x)=
   \left\{
   \begin{array}{cl} 
	\overline{p}_1(x),\quad &x\in[x_0,x_1],\\
	\overline{p}_2(x),\quad &x\in[x_1,x_2],\\
	\vdots&\\
	\overline{p}_i(x),\quad&x\in[x_{i-1},x_i], \\
	\vdots&\\
	\overline{p}_N(x),\quad &x\in[x_{N-1},x_{N}],
   \end{array}
   \right.
\end{equation}
where
\begin{equation}
  \overline{p}_i(x)= 
  \left\{ 
    \begin{array}{cc}
        \dfrac{\overline{y}_i-\overline{y}_{i-1}}{x_i-x_{i-1}}(x-x_{i-1})+\overline{y}_{i-1},\quad &x\in[x_{i-1},x_i], \\\\
         0,  \quad &x\notin[x_{i-1},x_i],
    \end{array}
  \right.         
\end{equation}
and $i=1,2,\ldots,N.$

\begin{theorem}
	The following estimate of the error holds
	\begin{equation}
	    \max_{x\in[0,1]}\left|y(x)-\overline{P}(x)\right|\leqslant \left( C\ln^2N \right) /N^2,
	\end{equation}
where $y$ is the exact solution to the problem \eqref{problem1}--\eqref{problem2} and $\overline{P}$  is global numerical solution \eqref{globalSolution1}.	
\end{theorem}

\begin{proof}
	We divide this proof in three parts, $[0,\lambda],\:[\lambda,x_{N/4+1}]$ and $[x_{N/4}+1,1/2]$. The proof is analogues on $[1/2,1].$ The proof is based on the inequality $\Vert \overline{P}-y\Vert_{\infty}\leqslant \Vert \overline{P}-P\Vert_{\infty}+\Vert P-y\Vert_{\infty},$ and a theorem on the interpolation error and its corollaries. For our purpose we use \cite[Example 8.12]{kress1998numerical}.  By $P$ is designated a piecewise polynomial obtained in the same way like $\overline{P},$ but $P$ passes trough the points with the coordinates $(x_{i-1},y(x_{i-1})),\:(x_{i},y(x_{i})),\:i=1,2,\ldots,N;$ instead of $(x_{i-1},\overline{y}_{i-1}),\:(x_{i},\overline{y}_{i+1}),\:i=1,2,\ldots,N,$

	\begin{equation}\label{globalSolution1a}
	P(x)=
	\left\{
	\begin{array}{cl} 
	p_1(x),\quad &x\in[x_0,x_1],\\
	p_2(x),\quad &x\in[x_1,x_2],\\
	\vdots&\\
	p_i(x),\quad&x\in[x_{i-1},x_i], \\
	\vdots&\\
	p_N(x),\quad &x\in[x_{N-1},x_{N}],
	\end{array}
	\right.
	\end{equation}
	where
	\begin{equation}
	p_i(x)= 
	\left\{ 
	\begin{array}{cc}
	\dfrac{ y_i- y_{i-1}}{x_i-x_{i-1}}(x-x_{i-1})+ y_{i-1},\quad &x\in[x_{i-1},x_i], \\\\
	0,  \quad &x\notin[x_{i-1},x_i],
	\end{array}
	\right.         
	\end{equation}
	and $i=1,2,\ldots,N.$ \\
	
	\noindent Taking in account the constructions \eqref{globalSolution1}, \eqref{globalSolution1a} and Theorem \ref{thConvergence1} holds 
	\begin{equation}\label{globalSolution2a}
	\Vert P-\overline{P}\Vert_{\infty}\leqslant \left( C\ln^2N \right) /N^2,\quad x\in[0,1].
	\end{equation} 
	
	The first part is on the subinterval $[0,\lambda],$ this one corresponds with the mesh when $i=1,2,\ldots, N/4.$ Here, the mesh is equidistant i.e. $h_{i-1}=h_i,$ and $h_i=\mathcal{O}\left(\varepsilon \ln N/N\right).$ Using Theorem \ref{theoremDecomposition}, \cite[Example 8.12]{kress1998numerical}, $h_i=\mathcal{O}\left(\varepsilon \ln N/N\right)$ we have that 
	\begin{align}\label{globalSolution2b}
	 |y(x)-p_i(x)|
	     \leqslant& \frac{h^2_i}{8}\max_{\xi \in[x_{i-1},x_{i}]}\left|y''(\xi )\right| 
	       \leqslant C_1 \frac{\varepsilon^2 \ln^2N}{N^2}\max_{\xi\in[x_{i-1},x_{i}]}\left| s''(\xi)+r''(\xi) \right| \nonumber\\
	     \leqslant& C_2 \frac{\varepsilon^2 \ln^2N}{N^2}\max_{\xi\in[x_{i-1},x_{i}]}\left| 
	               \varepsilon^{-2}\left(e^{-\frac{\xi}{\varepsilon}\sqrt{m}}+e^{-\frac{(\xi-1)}{\varepsilon}\sqrt{m}} \right)    +r''(\xi) \right| \nonumber\\
	     \leqslant& C_2 \frac{\varepsilon^2 \ln^2N}{N^2}(\varepsilon^{-2}+C_3) \leqslant\frac{C\ln^2 N}{N^2},\:i=1,2,\ldots,N/4.          
	\end{align}
	
	The remain of the proof, i.e. for  $x\in[\lambda,x_{N/4+1}]\cup[x_{N/4+1},1/2]$ which corresponds with the mesh for $i=N/4,N/4+1,\ldots, N/2,$ we repeat from \cite{samir2017construction}. \\

	\noindent For $i=N/4+1,\ldots, N/2,$ the mesh isn't equidistant but holds  $h_{i}=\mathcal{O}(1/N).$ According to the Theorem \ref{theoremDecomposition}, to the Theorem \eqref{thConvergence1}, \cite[Example]{kress1998numerical} and the features of the  mesh we obtain 
	\begin{equation}\label{globalSolution2c}
	\left|y(x)-p_i(x)\right|\leqslant \frac{h^2_i}{8}\max_{\xi\in[x_{N/4+1},1/2]}\left|y''(\xi)\right|\leqslant\frac{C}{N^2}.
	\end{equation}     
	On $[\lambda,x_{N/4}+1],$ according to the Theorem \ref{theoremDecomposition} we obtain 
	\begin{align*}
	y-p_i(x)=& y-\frac{y_i-y_{i-1}}{x_i-x_{i-1}}(x-x_{i-1})+y_{i-1} \nonumber\\
	        =& s-\frac{s_i-s_{i-1}}{x_i-x_{i-1}}(x-x_{i-1})+s_{i-1}+r-\frac{r_i-r_{i-1}}{x_i-x_{i-1}}(x-x_{i-1})+r_{i-1}.
	\end{align*}
	For the layer component $s,$ based on the estimate \eqref{slojna}, we have
	\begin{align}
	\left|s-\frac{s_i-s_{i-1}}{x_i-x_{i-1}}(x-x_{i-1})+s_{i-1} \right|
	     \leqslant |s|+|s_{i+1}-s_i|+|s_i|
	        \leqslant C_1\left(e^{-\frac{x_{i-1}}{\varepsilon}\sqrt{m}}+e^{-\frac{x_{i-1}-1}{\varepsilon}\sqrt{m}} \right)
	     \leqslant \frac{C}{N^2}. \label{globalSolution2d}
	\end{align}
	For the regular component $r,$ we apply again the estimate from \cite[Example 8.12]{kress1998numerical}, the estimate \eqref{regularna}, and we have that 
	\begin{equation}\label{globalSolution2e}
	\left|r-\frac{r_i-r_{i-1}}{x_i-x_{i-1}}(x-x_{i-1})+r_{i-1} \right|
	    \leqslant \frac{h^2_{i-1}}{8}\max_{\xi\in[x_{i-1},x_i]}\left|r''(\xi)\right|\leqslant \frac{C}{N^2}.
	\end{equation}
	Collecting \eqref{globalSolution2a}, \eqref{globalSolution2b},  \eqref{globalSolution2c}, \eqref{globalSolution2d} and \eqref{globalSolution2e}, this theorem has been proven. 	
\end{proof}
\paragraph{Cubic spline} In the numerical experiments we will use a natural cubic spline as a global solution. We construct it in the way as follows: design the natural cubic spline by $C,$ 
\begin{equation}\label{globalSpline1}
   C(x)=C_i(x),\:x\in[x_i,x_{i+1}],\:i=0,1,\ldots,N-1,
\end{equation} 
where $C_i$ are the cubic functions
\begin{equation}\label{globalSpline2}
 C_i(x)=M_i\frac{(x_{i+1}-x)^3}{6h_{i+1}}+M_{i+1}\frac{(x-x_i)^3}{6h_{i+1}}+\left[\frac{\overline{y}_{i+1}-\overline{y}_i}{h_{i+1}}-\frac{h_{i+1}}{6}(M_{i+1}-M_i)\right](x-x_i)
         +\overline{y}_i-M_i\frac{h^2_{i+1}}{6},
\end{equation}
the moments $M_i:=C''_i(x_i),\:i=1,N-1$ we get from the system 
\begin{equation}\label{globalSpline3}
   \frac{h_i}{6}M_{i-1}+\frac{h_i+h_{i+1}}{3}M_i+\frac{h_{i+1}}{6}M_{i+1}=\frac{\overline{y}_{i+1}-\overline{y}_i}{h_{i+1}}-\frac{\overline{y}_i-\overline{y}_{i-1}}{h_i},\:i=1,2,\ldots,N-1,
\end{equation}
and $M_0:=C_0''(x_0)=0,$ $M_N:=C_{N-1}''(x_N)=0.$

\section{Numerical experiments}
In this section we conduct numerical experiments in order to confirm the theoretical results, i.e. to confirm the accuracy of the different scheme \eqref{shema2} on the meshes \eqref{meshShishkin1}, \eqref{meshShishkin2}, \eqref{meshVulanovic1} and \eqref{meshLiseikin}.

\begin{example}
	We consider the following  boundary value problem 
	
	\begin{equation}
	\varepsilon ^2y''=y+\cos^2\pi x+2\varepsilon ^2\pi^2\cos^2\pi x\quad\text{on}\quad\left(0,1\right),
	\label{example1}
	\end{equation}
	\begin{equation}
	y(0)=y(1)=0.
	\label{problem}
	\end{equation}
	The exact solution of this problem is
	\begin{equation}
	y(x)=\dfrac{e^{-\frac{x}{\varepsilon}}+e^{\frac{x}{\varepsilon}}}{1+e^{-\frac{1}{\varepsilon}}}-\cos^2\pi x.
	\label{problem3}
	\end{equation} 	
The nonlinear  system was solved using the initial condition $y_0=-0.5$ and the value of the constant $\gamma=1.$ Because of the fact that the exact solution is known, we compute the error $E_N$ and the rate of convergence Ord in the usual way
\begin{equation}
E_N=\Vert y-\overline{y}^{N}\Vert_{\infty},\: {\rm Ord}=\frac{\ln E_N-\ln E_{2N}}{\ln(2k/(k+1))},({\rm Shishkin}),\:\:{\rm Ord}=\frac{\ln E_N-\ln E_{2N}}{\ln 2},\:({\rm Bakhvalov, Liseikin})
\end{equation}
where $N=2^{k},$ $k=4,5,\ldots,12,$ and $y$ is the exact solution of the problem \eqref{problem1}--\eqref{problem3}, while $\overline{y}^N$ an appropriate numerical solution of \eqref{operator1}. The graphics of the numerical and exact solutions, for various values of the parameter $\varepsilon$ are on Figure  \ref{figure1} (left), while fragments of these solutions are on Figure \ref{figure1} (right). The values of $E_N$ and Ord are in Tables \ref{table1}. 	The graphics of the exact and global solution obtained by using a linear spline, and the corresponding error are shown on Figure \ref{figure2}, while the graphics of the exact and global solution obtained by using a natural cubic spline, and the corresponding error are shown on Figure \ref{figure3}. 
\end{example}

\begin{figure}[!h]\centering
	\includegraphics[scale=0.34]{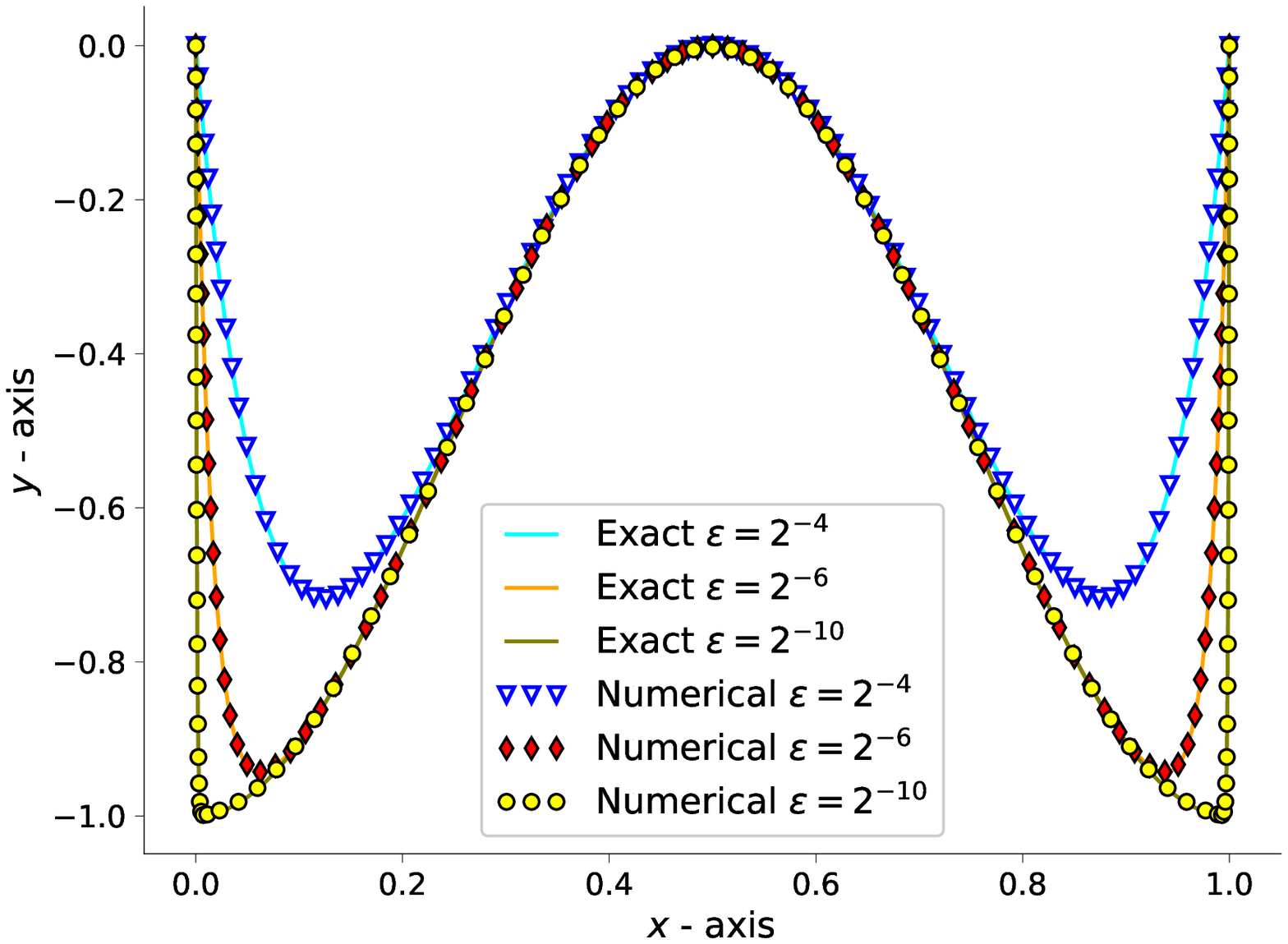}\vspace{.5cm}
	\includegraphics[scale=0.34]{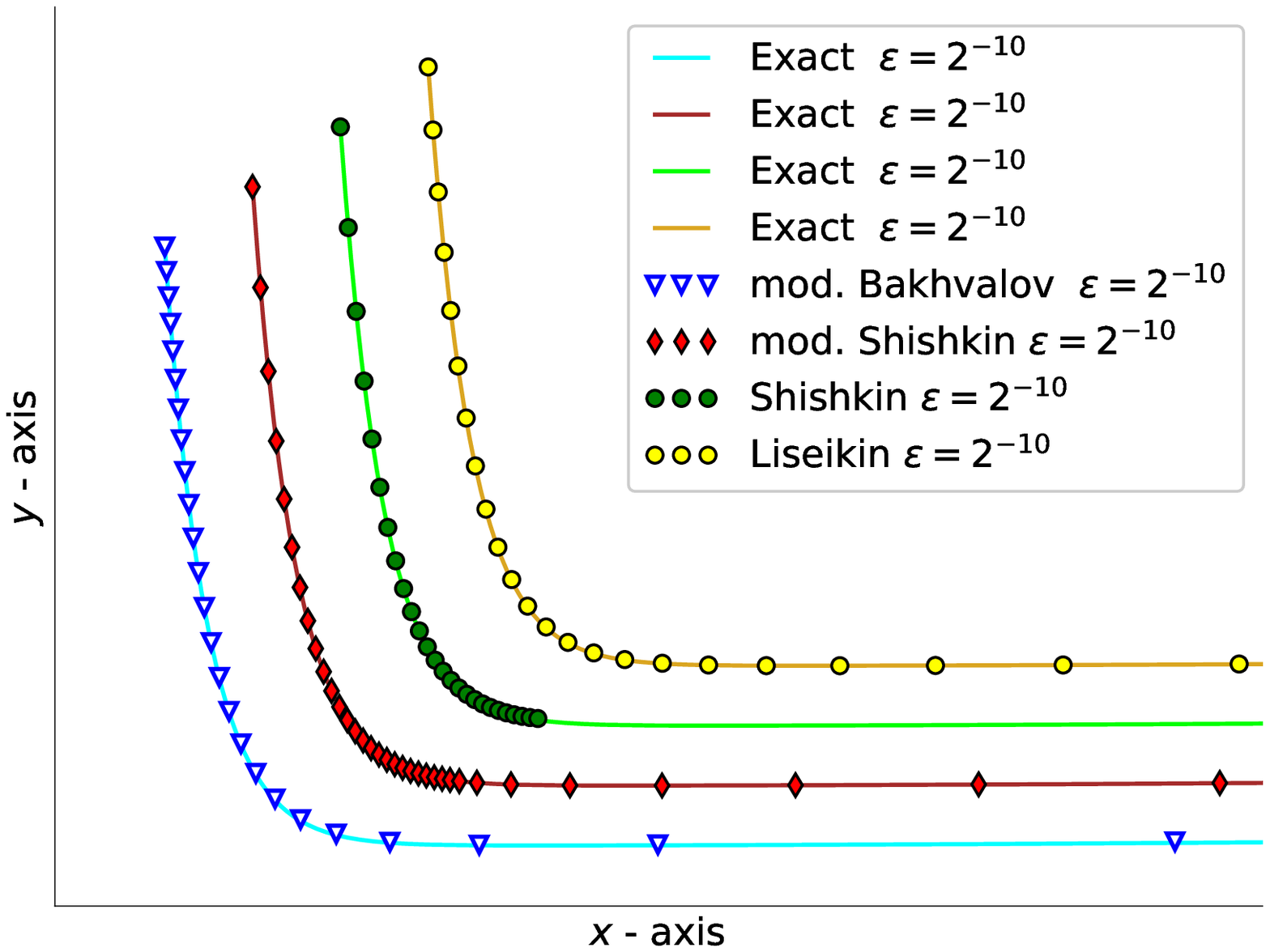}\vspace{.5cm}
	\caption{Exact and numerical solutions (left), layer near $x=0$ (right) }
	 \label{figure1}
\end{figure}

\newpage

\begin{figure}[!h] \centering
	\includegraphics[scale=0.34]{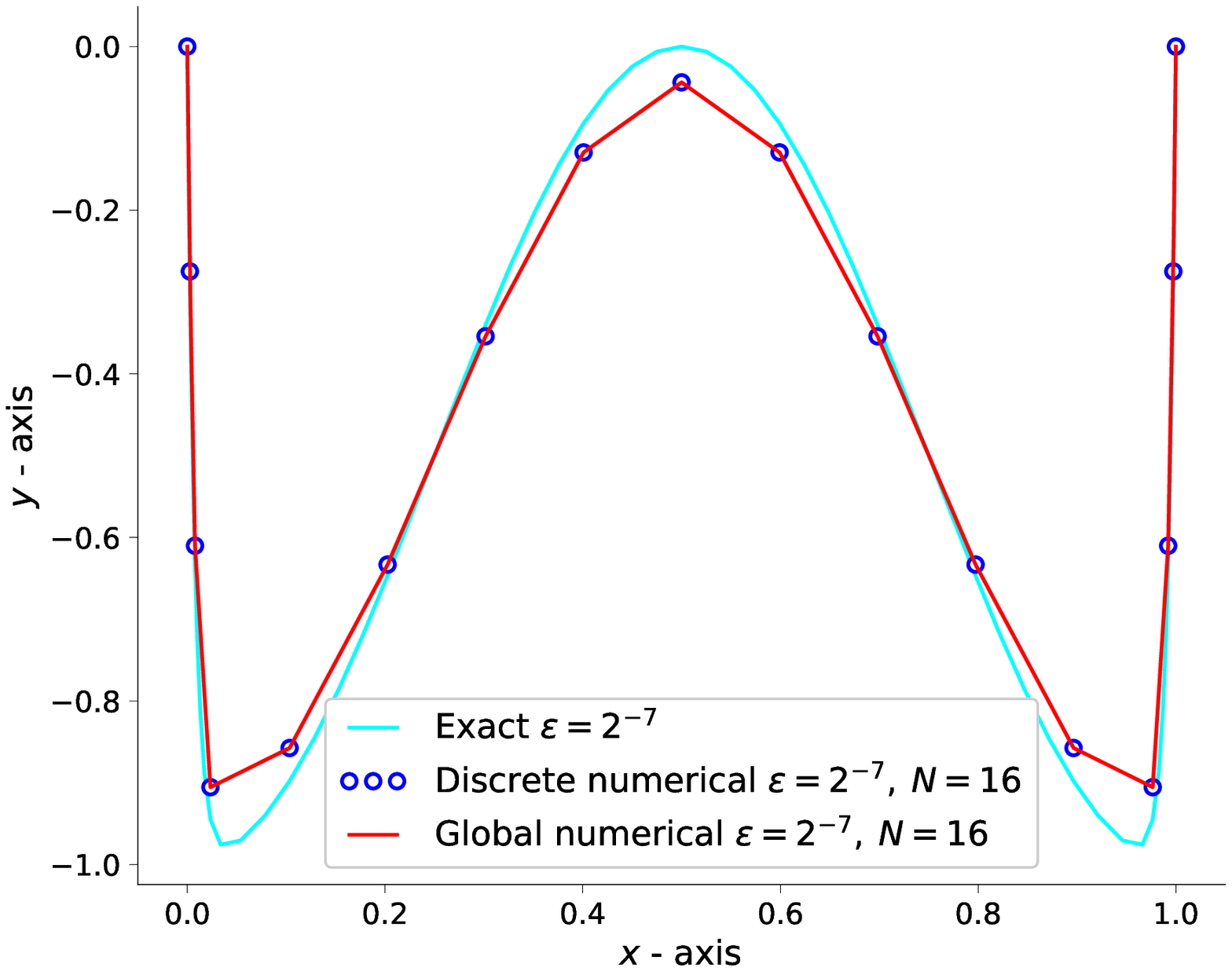}
	\includegraphics[scale=0.34]{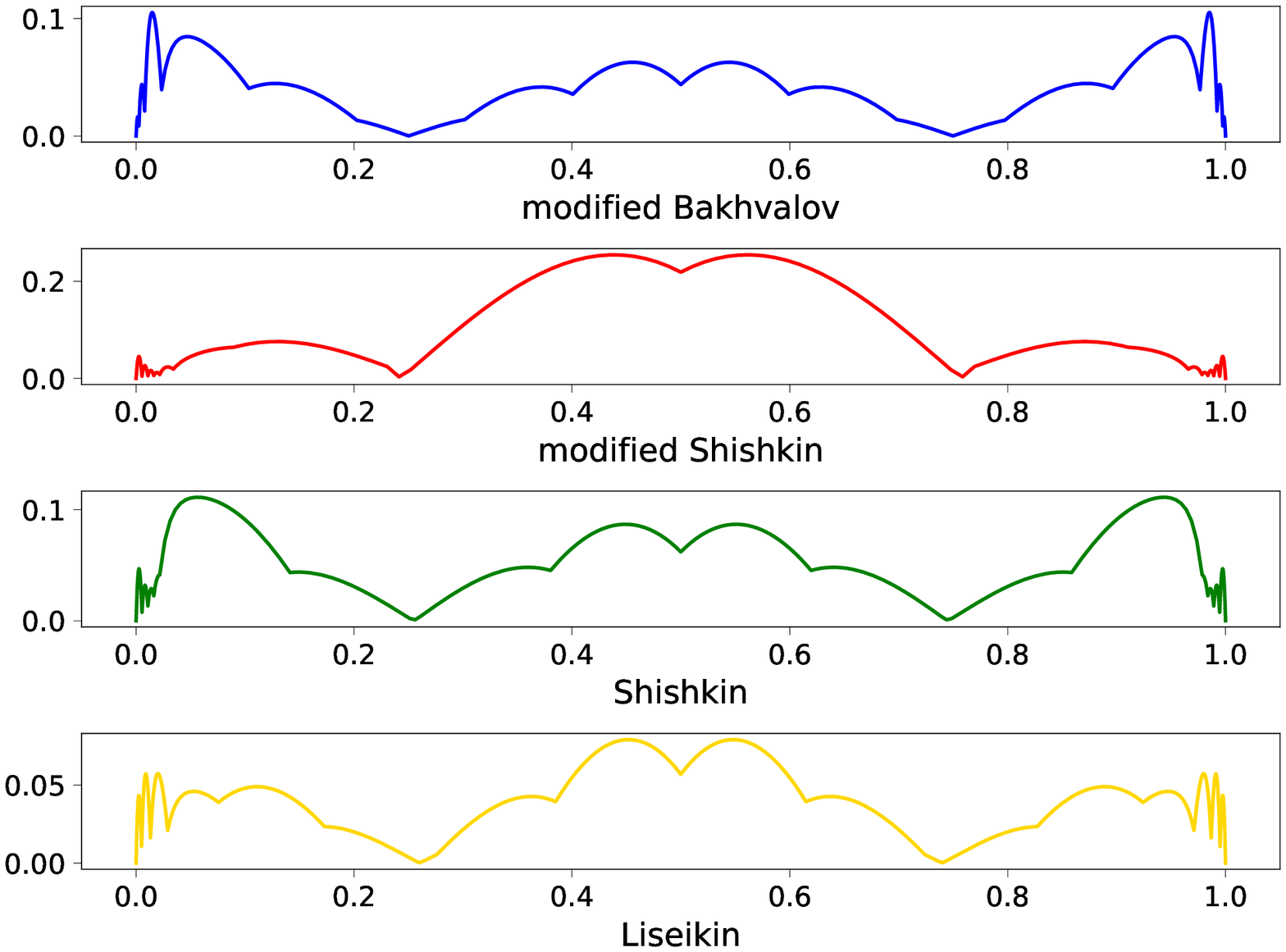}
	
	\includegraphics[scale=0.34]{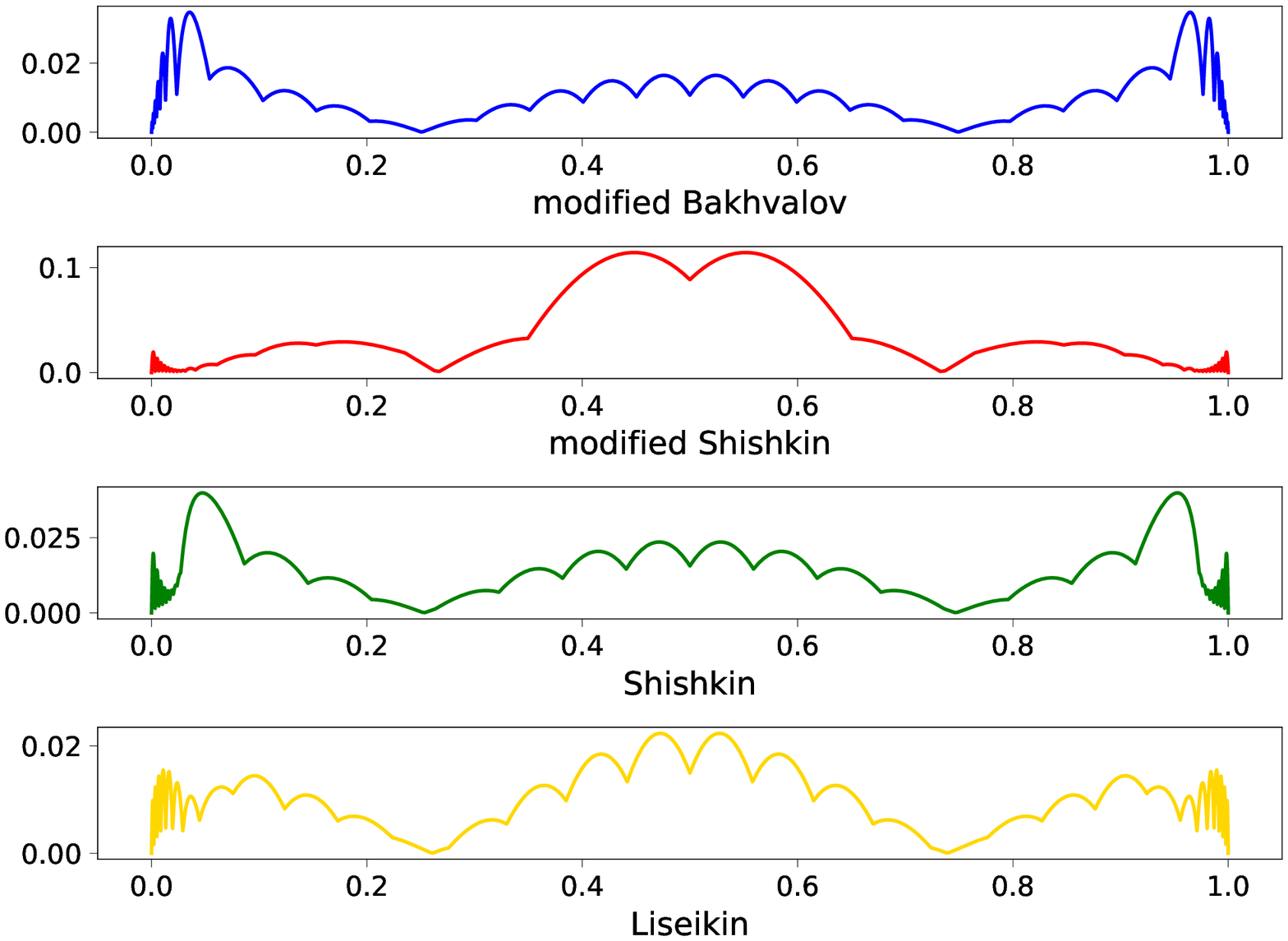}
	\includegraphics[scale=0.34]{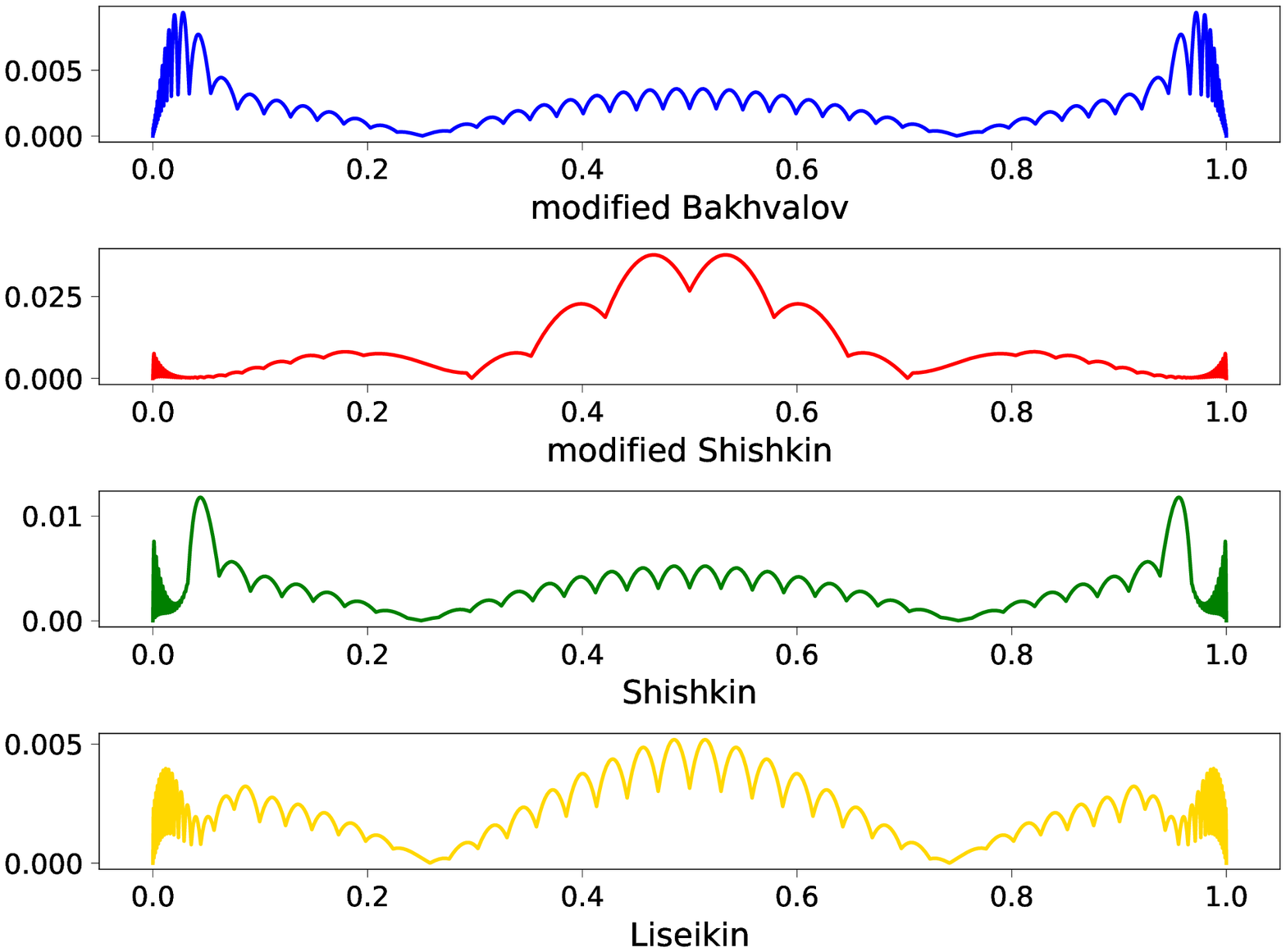}
	\caption{Exact, discrete and global numerical solutions (left up), error  (right up--$N=16$, left down--$N=32$, right down--$N=64$) }
	\label{figure2}
\end{figure}

\begin{figure}[!h] \centering
	\includegraphics[scale=0.34]{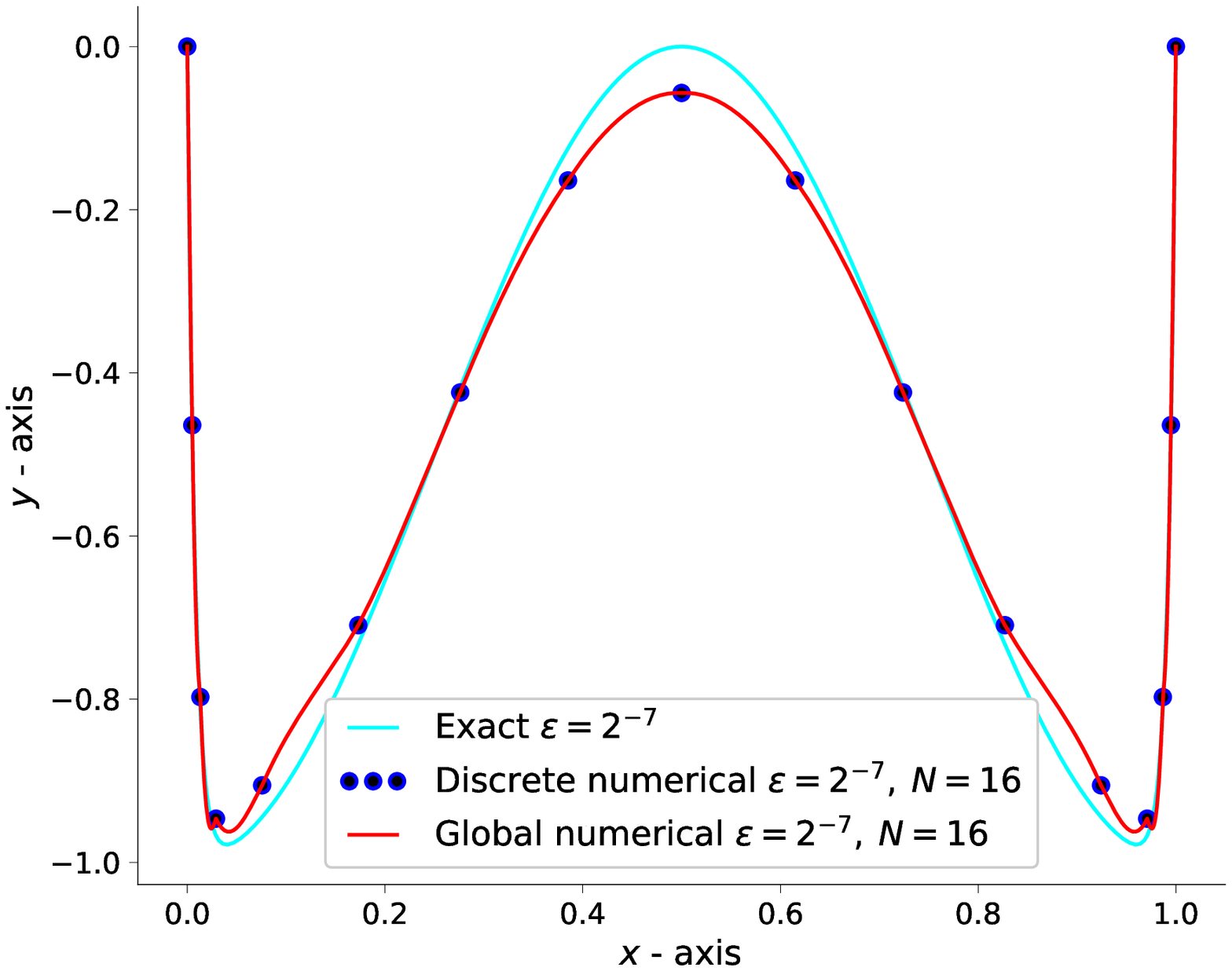}
	\includegraphics[scale=0.34]{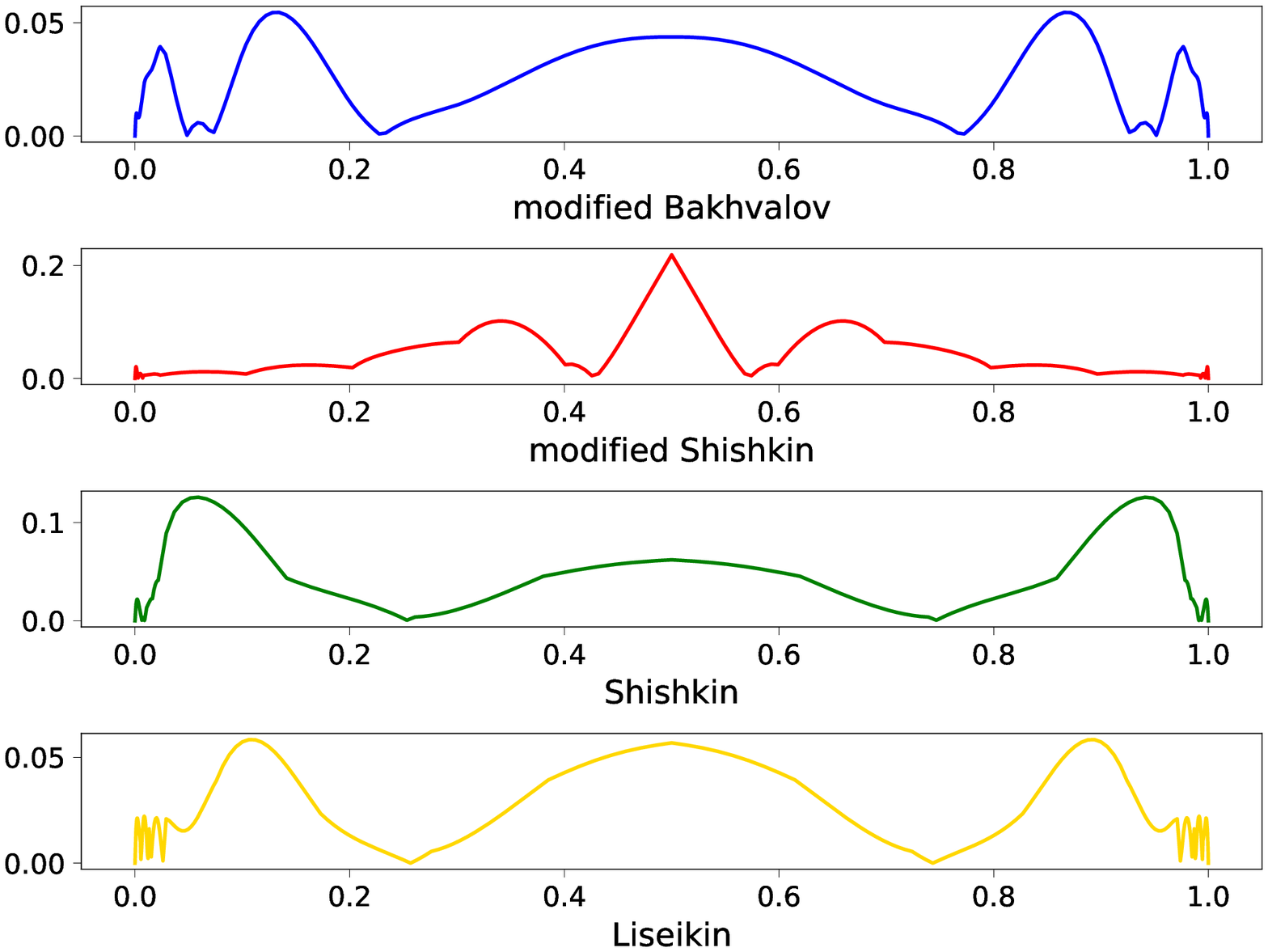}
	
	\includegraphics[scale=0.34]{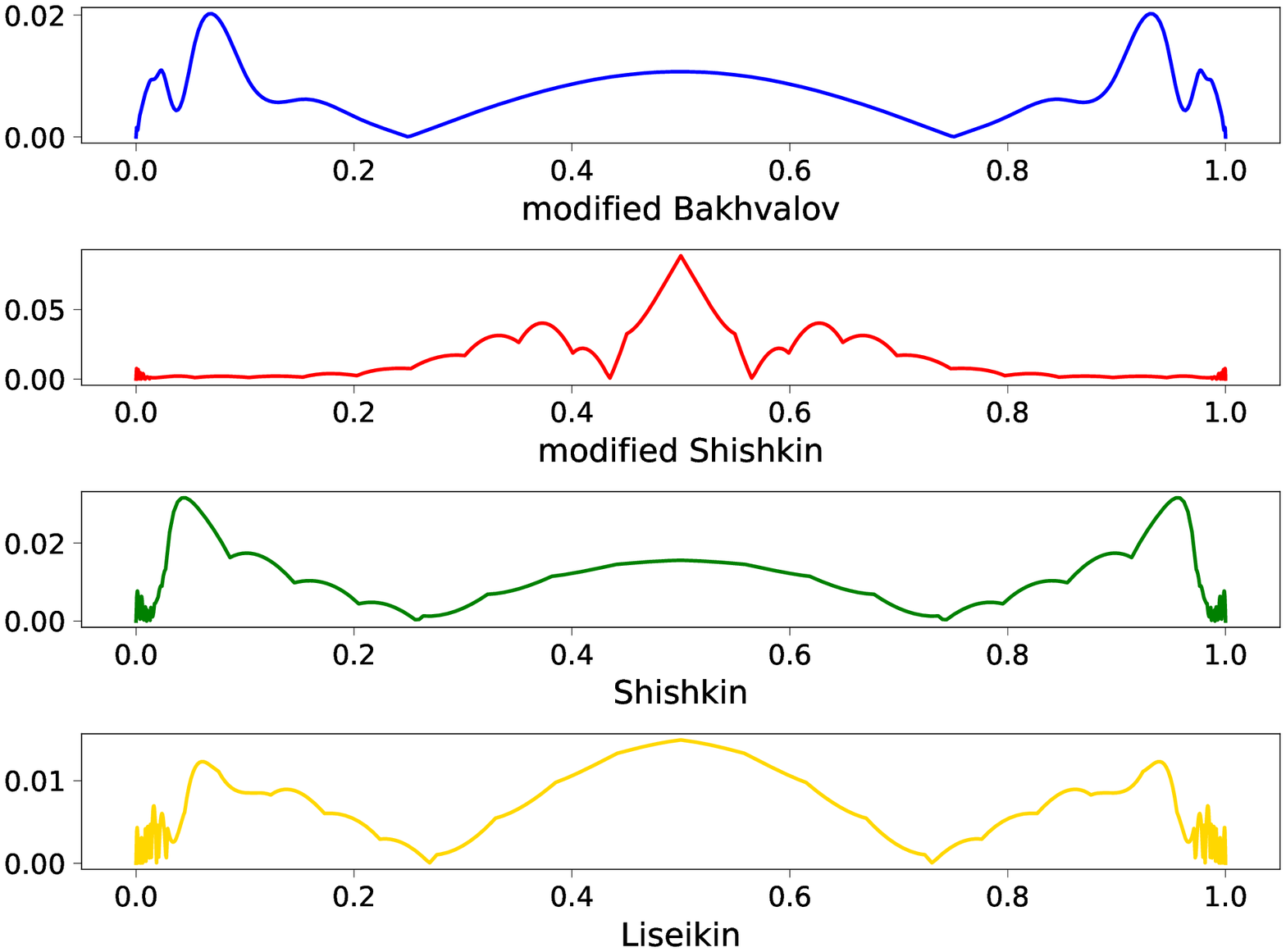}
	\includegraphics[scale=0.34]{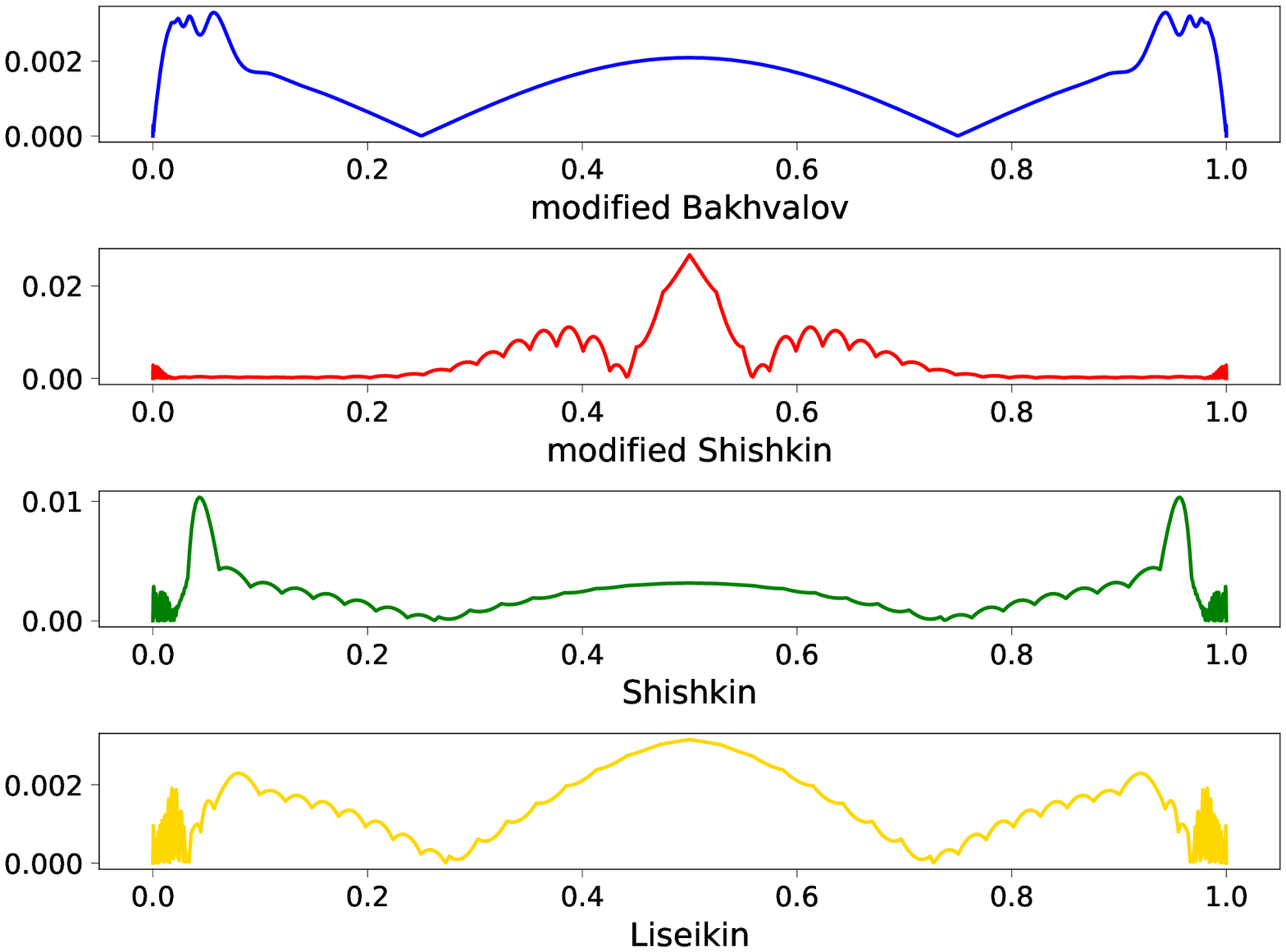}
	\caption{Exact, discrete and global numerical solutions (left up), error  (right up--$N=16$, left down--$N=32$, right down--$N=64$) }
	\label{figure3}
\end{figure} 

\newpage

\begin{table*}[!h]\centering	\tiny
	\begin{tabular}{c cc cc cc cc cc cc}
		\toprule 
		
		\multicolumn{1}{c}{} 
		& \multicolumn{2}{c}{$2^{-3}$}   
		& \multicolumn{2}{c}{$2^{-5}$} 
		& \multicolumn{2}{c}{$2^{-10}$} 
		& \multicolumn{2}{c}{$2^{-20}$}  
		& \multicolumn{2}{c}{$2^{-30}$}
		& \multicolumn{2}{c}{$2^{-40}$}\\
		
		\cmidrule(r{\tabcolsep}){2-3} 
		\cmidrule(r{\tabcolsep}){4-5} 
		\cmidrule(r{\tabcolsep}){6-7}
		\cmidrule(r{\tabcolsep}){8-9}
		\cmidrule(r{\tabcolsep}){10-11}
		\cmidrule(r{\tabcolsep}){12-13}
		
		\multicolumn{1}{c}{$N$}
		& \multicolumn{1}{c}{$E_n$}
		& \multicolumn{1}{c}{Ord }
		& \multicolumn{1}{c}{$E_n$}
		& \multicolumn{1}{c}{Ord }
		& \multicolumn{1}{c}{$E_n$}
		& \multicolumn{1}{c}{Ord }
		& \multicolumn{1}{c}{$E_n$}
		& \multicolumn{1}{c}{Ord }
		& \multicolumn{1}{c}{$E_n$}
		& \multicolumn{1}{c}{Ord }\\
		\midrule
		
		$2^{4}$   &5.277e-2 &3.09 &1.015e-1 &2.08 &1.255e-1 & 2.67 &1.278e-1 &2.65 &1.278e-1 &2.65 &1.278e-1 &2.65 \\ 
		$2^{5}$   &1.234e-2 &2.89 &3.811e-2 &2.83 &3.570e-2 & 2.65 &3.666e-2 &2.64 &3.666e-2 &2.64 &3.666e-2 &2.64 \\ 
		$2^{6}$   &2.819e-3 &2.75 &8.956e-3 &2.87 &9.194e-3 & 2.20 &9.515e-3 &2.26 &9.515e-3 &2.26 &9.515e-3 &2.26 \\ 
		$2^{7}$   &6.374e-4 &2.67 &1.900e-3 &2.74 &2.804e-3 & 1.99 &2.803e-3 &1.99 &2.803e-3 &1.99 &2.804e-3 &1.99 \\ 
		$2^{8}$   &1.429e-4 &2.61 &4.099e-4 &2.61 &9.198e-4 & 2.00 &9.196e-4 &2.00 &9.196e-4 &2.00 &9.196e-4 &2.00 \\ 
		$2^{9}$   &3.175e-5 &2.57 &9.104e-5 &2.52 &2.911e-4 & 2.00 &2.911e-4 &2.00 &2.911e-4 &2.00 &2.911e-4 &2.00 \\ 
		$2^{10}$  &6.988e-6 &2.54 &2.069e-5 &2.42 &8.987e-5 & 2.00 &8.986e-5 &2.00 &8.986e-5 &2.00 &8.986e-5 &2.00 \\
		$2^{11}$  &1-522e-6 &2.52 &4.851e-6 &2.33 &2.719e-5 & 2.00 &2.719e-5 &2.00 &2.719e-5 &2.00 &2.719e-5 &2.00 \\
		$2^{12}$  &3.286e-7 &-    &1.180e-6 & -   &8.091e-6 & -    &8.090e-6 & -   &8.090e-6 &-    &8.090e-6 & - \\
		\midrule \\
		& & & & & & mesh \eqref{meshShishkin1}& & & & & &\\
		\midrule  
		
		$2^{4}$   &2.012e-1 &2.88 &1.944e-1 &2.29 &2.356e-1 &1.80 &2.408e-1 &1.76 &2.408e-1 &1.766 &2.408e-1 &1.76 \\ 
		$2^{5}$   &5.184e-2 &3.06 &6.598e-2 &3.06 &1.010e-1 &2.20 &1.049e-1 &2.17 &1.050e-1 &2.17  &1.050e-1 &2.17 \\ 
		$2^{6}$   &1.082e-2 &3.10 &1.377e-2 &3.13 &3.276e-2 &2.36 &3.450e-2 &2.34 &3.450e-2 &2.34  &3.450e-2 &2.34 \\ 
		$2^{7}$   &2.129e-3 &2.96 &2.547e-3 &2.76 &9.157e-3 &2.40 &9.768e-3 &2.37 &9.769e-3 &2.37  &9.769e-3 &2.37 \\ 
		$2^{8}$   &4.048e-4 &2.94 &5.413e-4 &2.57 &2.381e-3 &2.42 &2.581e-3 &2.36 &2.581e-3 &2.36  &2.581e-3 &2.36 \\ 
		$2^{9}$   &7.453e-5 &2.93 &1.226e-4 &2.50 &5.907e-4 &2.51 &6.619e-4 &2.33 &6.620e-4 &2.33  &6.620e-4 &2.33 \\ 
		$2^{10}$  &1.327e-5 &2.93 &2.818e-5 &2.45 &1.343e-4 &2.82 &1.674e-4 &2.30 &1.674e-4 &2.30  &1.674e-4 &2.30 \\
		$2^{11}$  &2.295e-6 &2.91 &6.478e-6 &2.43 &2.487e-5 &2.92 &4.211e-5 &2.28 &4.211e-5 &2.28  &4.212e-5 &2.28 \\
		$2^{12}$  &3.929e-7 & -   &1.484e-6 &-    &4.233e-6 & -   &1.055e-5 &-    &1.055e-5 & -    &1.055e-5 &-  \\
		\midrule \\
		& & & & & & mesh \eqref{meshShishkin2}& & & & & &\\
		\midrule
		
		$2^{4}$   &5.240e-3 &2.03 &3.038e-2 &1.97 &5.847e-2 &1.89 &6.790e-2 &1.87 &6.822e-2 &1.86 &6.823e-2 &1.86 \\ 
		$2^{5}$   &1.282e-3 &2.00 &7.750e-3 &1.94 &1.577e-2 &1.98 &1.857e-2 &1.97 &1.867e-2 &1.97 &1.867e-2 &1.96 \\ 
		$2^{6}$   &3.186e-4 &2.00 &2.017e-3 &1.96 &4.009e-3 &1.89 &4.754e-3 &1.99 &4.779e-3 &1.99 &4.780e-3 &1.99 \\ 
		$2^{7}$   &7.954e-5 &2.00 &5.163e-4 &1.99 &1.076e-3 &1.68 &1.195e-3 &2.00 &1.202e-3 &2.00 &1.202e-3 &2.00 \\ 
		$2^{8}$   &1.987e-5 &2.00 &1.295e-4 &2.00 &3.355e-4 &2.08 &2.993e-4 &2.00 &3.009e-4 &2.00 &3.010e-4 &2.00 \\ 
		$2^{9}$   &4.969e-6 &2.00 &3.246e-5 &2.00 &7.912e-5 &2.54 &7.487e-5 &2.00 &7.527e-5 &2.00 &7.528e-5 &2.00 \\ 
		$2^{10}$  &1.242e-6 &2.00 &8.117e-6 &2.00 &1.357e-5 &2.00 &1.872e-5 &1.99 &1.882e-5 &2.00 &1.882e-5 &2.00 \\
		$2^{11}$  &3.105e-7 &2.00 &2.029e-6 &2.00 &3.397e-6 &2.00 &4.704e-6 &1.85 &4.705e-6 &2.00 &4.706e-6 &2.00 \\
		$2^{12}$  &7.764e-8 &-    &5.073e-7 & -   &8.494e-7 &-    &1.300e-6 & -   &1.176e-6 & -   &1.176e-6 & - \\
		\midrule \\
		& & & & & & mesh \eqref{meshVulanovic1}& & & & & &\\
		
		\midrule
		
		$2^{4}$   &6.452e-3 &2.01 &1.209e-2 &2.43 &3.055e-2 &1.96 &3.593e-2 &1.95 &3.654e-2 &1.94 &3.660e-2 &1.94 \\ 
		$2^{5}$   &1.593e-3 &2.00 &2.234e-3 &2.18 &7.873e-3 &1.69 &9.332e-3 &1.97 &9.496e-3 &1.85 &9.513e-3 &1.97 \\ 
		$2^{6}$   &3.968e-4 &2.00 &4.897e-4 &2.05 &2.444e-3 &1.78 &2.355e-3 &2.00 &2.397e-2 &2.00 &2.401e-3 &2.00 \\ 
		$2^{7}$   &9.913e-5 &2.00 &1.177e-4 &2.01 &7.102e-4 &1.96 &5.902e-4 &2.00 &6.000e-4 &2.00 &6.017e-4 &2.00 \\ 
		$2^{8}$   &2.477e-5 &2.00 &2.913e-5 &2.00 &1.819e-4 &2.48 &1.476e-4 &2.00 &1.502e-4 &2.00 &1.505e-4 &2.00 \\ 
		$2^{9}$   &6.194e-6 &2.00 &7.264e-6 &2.00 &3.255e-5 &3.27 &4.469e-5 &2.00 &3.757e-5 &2.00 &3.764e-5 &2.00 \\ 
		$2^{10}$  &1.548e-6 &2.00 &1.814e-6 &2.00 &3.355e-6 &1.99 &1.354e-5 &2.00 &9.393e-6 &2.00 &9.410e-6 &2.00 \\
		$2^{11}$  &3.871e-7 &2.00 &4.536e-7 &2.00 &8.462e-7 &1.54 &3.867e-6 &2.00 &2.348e-6 &2.00 &2.352e-6 &2.00 \\
		$2^{12}$  &9.678e-8 & -   &1.134e-7 & -   &2.905e-7 & -   &1.196e-6 & -   &6.604e-7 & -   &5.881e-7 & -\\
		\midrule \\
		& & & & & & mesh \eqref{meshLiseikin}& & & & & &\\

		\bottomrule
	\end{tabular}
	\caption{Values of $E_N$ and Ord}
	\label{table1}
\end{table*} 

\newpage
\section{Conclusion}
In the present paper we performed the construction of a numerical solution for the one--dimensional singularly--perturbed reaction--diffusion boundary--value problem. The class of different schemes was constructed, and we proved the existence and uniqueness of the discrete numerical solution. After that, we proved $\varepsilon$--uniformly convergence of the constructed class of different schemes on the modified Shishkin mesh of order 2. A global numerical solution was constructed based on a linear spline and proved that the order of the error  value is $\mathcal{O}\left(\ln^2N/N^2\right).$ The numerical experiments at the end of the paper confirm the theoretical results. The results obtained by using a global numerical solution based on a natural cubic spline and the Shishkin, the modified Bakhvalov and last but not least the Liseikin mesh are included in the numerical experiments. Although, the theoretical analysis for these meshes wasn't done, the results suggest that the order  of convergence is 2 for all of them. Especially,  good results have been achieved by using the Liseikin mesh.

\end{document}